\documentclass[a4paper,10pt]{amsart}
\usepackage{amsmath,amssymb}
\usepackage{amscd}
\newtheorem{thm}{Theorem}[section]

\newtheorem{lemma}[thm]{Lemma}

\theoremstyle{remark}

\newcommand{\id}{{\rm{id}}}
\newcommand{\Ad}{{\rm{Ad}}}

\newcommand{\BN}{\mathbf N}

\newcommand{\BC}{\mathbf C}
\newcommand{\BZ}{\mathbf Z}
\newcommand{\BK}{\mathbf K}

\newcommand{\BT}{\mathbf T}
\newcommand{\la}{\langle}
\newcommand{\ra}{\rangle}

\newcommand{\Pic}{{\rm{Pic}}}

\newcommand{\Aut}{{\rm{Aut}}}
\newcommand{\Int}{{\rm{Int}}}
\newcommand{\Ima}{{\rm{Im}}}
\newcommand{\Ker}{{\rm{Ker}}}

\newtheorem{Def}{Definition}[section]

\title{The Picard groups of inclusions of $C^*$-algebras induced by equivalence bimodules}
\author{Kazunori Kodaka}
\address{Department of Mathematical Sciences, Faculty of Science, Ryukyu
\endgraf
University, Nishihara-cho, Okinawa, 903-0213, Japan}
\address{\sl{E-mail address}: \rm{kodaka@math.u-ryukyu.ac.jp}}



\begin{document}
\begin{abstract}
Let $A$ and $B$ be $\sigma$-unital $C^*$-algebras and $X$ and $Y$ an $A-A$-equivalence
bimodule and a $B-B$-equivalence bimodule, respectively. Also, let $A\rtimes_X \BZ$ and
$B\rtimes_Y \BZ$ be the crossed products of $A$ and $B$ by $X$ and $Y$, respectively.
Furthermore, let $A\subset A\rtimes_X \BZ$ and $B\subset B\rtimes_Y \BZ$ be the inclusions of
$C^*$-algebras induced by $X$ and $Y$, respectively.
We suppose that $A' \cap M(A\rtimes_X \BZ)=\BC 1$. In this paper we shall show that the inclusions
$A\subset A\rtimes_X \BZ$ and $B\subset B\rtimes_Y \BZ$ are strongly Morita equivalent if and only if
there is an $A-B$-equivalence bimodule $M$ such that $Y\cong \widetilde{M}\otimes_A X \otimes_A M$
or $\widetilde{Y}\cong \widetilde{M}\otimes_A X \otimes_A M$ as $B-B$-equivalence bimodules,
where $\widetilde{M}$ and $\widetilde{Y}$ are the dual $B-A$-equivalence bimodule and
the dual $B-B$-equivalence bimodule of $M$ and $Y$, respectively.
Applying this result, we shall compute the Picard group of the inclusion $A\subset A\rtimes_X \BZ$
under the assumption that $A' \cap M(A\rtimes_X \BZ)=\BC 1$. 
\end{abstract}

\maketitle

\section{Introduction}\label{sec:intro} In the previous paper \cite {Kodaka:bundle}, we discussed
strong Morita equivalence for unital inclusions of unital $C^*$-algebras induced by involutive
equivalence bimodules. That is, let $A$ and $B$ be unital $C^*$-algebras and $X$ and $Y$ an
involutive $A-A$-equivalence bimodule and an involutive $B-B$-equivalence bimodule, respectively.
Let $C_X$ and $C_Y$ be unital $C^*$-algebras  induced by $X$ and $Y$ and $A\subset C_X$ and
$B\subset C_Y$ the unital inclusions of unital $C^*$-algebras. We suppose that $A' \cap C_X =\BC 1$.
In the paper \cite {Kodaka:bundle}, we showed that $A\subset C_X$ and $B\subset C_Y$ are
strongly Morita equivalent if and only if there is an $A-B$-equivalence bimodule $M$ such that
$Y\cong \widetilde{M}\otimes_A X\otimes_A M$ as $B-B$-equivalence bimodules. In the present paper,
we shall show the same result as above in the case of inclusions of $C^*$-algebras induced by
$\sigma$-unital $C^*$-algebra equivalence bimodules.
\par
Let $A$ and $B$ be $\sigma$-unital $C^*$-algebras and $X$ and $Y$ an $A-A$-equivalence bimodule
and a $B-B$-equivalence bimodule, respectively. Let $A\rtimes_X \BZ$ and $B\rtimes_Y \BZ$
be the crossed products of $A$ and $B$ by $X$ and $Y$, respectively, which are defined in Abadie,
Eilers and Exel \cite {AEE:crossed}. Then we have inclusions of $C^*$-algebras $A\subset A\rtimes_X \BZ$
and $B\subset B\rtimes_Y \BZ$. We note that $\overline{A(A\rtimes_X \BZ)}=A\rtimes_X \BZ$ and
$\overline{B(B\rtimes_Y \BZ)}=B\rtimes_Y \BZ$. We suppose that $A' \cap M(A\rtimes_X \BZ)=\BC 1$.
We shall show that $A\subset A\rtimes_X \BZ$ and $B\subset B\rtimes_Y \BZ$ are strongly Morita
equivalent if and only if there is an $A-B$-equivalence bimodule $M$ such that
$Y\cong \widetilde{M}\otimes_A X\otimes_A M$ or
$\widetilde{Y}\cong \widetilde{M}\otimes_A X\otimes_A M$ as $B-B$-equivalence bimodules,
where $\widetilde{M}$ and $\widetilde{Y}$ are the dual $B-A$-equivalence bimodule and the dual
$B-B$-equivalence bimodule of $M$ and $Y$, respectively. Applying the above result, we shall compute
the Picard group of the inclusion of $C^*$-algebras $A\subset A\rtimes_X \BZ$ under the assumption
that $A' \cap M(A\rtimes_X \BZ)=\BC 1$.
\section{Preliminaries}\label{sec:pre} 
Let $\BK$ be the $C^*$-algebra of all compact operators on a countably infinite dimensional Hilbert
space and $\{e_{ij}\}_{i, j\in \BN}$ its system of matrix units.
\par
For each $C^*$-algebra $A$, we denote
by $M(A)$ the multiplier $C^*$-algebra of $A$. Let $\pi$ be an isomorphism of $A$ onto a $C^*$-algebra
$B$. Then there is the unique strictly continuous isomorphism of $M(A)$ onto $M(B)$
extending $\pi$ by Jensen and Thomsen \cite [Corollary 1.1.15]{JT:KK}. We denote it by $\underline{\pi}$.
\par
Let $A$ and $B$ be $C^*$-algebras and $X$ an $A-B$-bimodule. We denote its left $A$-action and
right $B$-action on $X$ by $a\cdot x$ and $x\cdot b$ for any $a\in A$, $b\in B$, $x\in X$, respectively.
We denote by $\widetilde{X}$ the dual $B-A$-bimodule of $X$ and we denote by $\widetilde{x}$ the
element in $\widetilde{X}$ induced by $x\in X$.
Let $A\subset C$ and $B\subset D$ be inclusions of $C^*$-algebras.

\begin{Def}\label{def:int1} We say that $A\subset C$ and
$B\subset D$ are
\sl
isomorphic as inclusions of $C^*$-algebras
\rm
if there is an isomorphism $\pi$ of $C$ onto $D$ such that $\pi|_A$ is an isomorphism
of $A$ onto $B$.
\end{Def}

Also, we give the following definition:

\begin{Def}\label{def:int2} Let $\alpha$ nad $\beta$ be actions of a discrete group $G$ on $A$ and $B$,
respectively. We say that $\alpha$ and $\beta$ are
\sl
strongly Morita equivalent with respect to $(X, \lambda)$
\rm
if there are an $A-B$-equivalence bimodule $X$ and a linear automorphism $\lambda$ of $X$ satisfying
the following:
\newline
(1) $\alpha_t ({}_A \la x, y \ra )={}_A \la \lambda_t (x)\, , \, \lambda_t (y) \ra$,
\newline
(2) $\beta_t (\la x, y \ra_B )=\la \lambda_t (x) \, , \, \lambda_t (y) \ra_B$
\newline
for any $x, y\in X$, $t\in G$.
 \end{Def}
 
 Then we have the following:
 $$
 \lambda_t (a\cdot x)=\alpha_t (a)\cdot \lambda_t (x) , \quad
 \lambda_t (x\cdot b)=\lambda_t (x)\cdot \beta_t (b)
 $$
 for any $a\in A$, $b\in B$, $x\in X$ and $t\in G$.
 \par
Let $A$ and $B$ be $C^*$-algebras and $\pi$ an isomorphism of
$B$ onto $A$. We construct an $A-B$-equivalence bimodule $X_{\pi}$ as follows: Let $X_{\pi}=A$ as
$\BC$-vector spaces. For any $a\in A$, $b\in B$, $x, y\in X_{\pi}$,
\begin{align*}
& a\cdot x=ax \, , \quad x\cdot b=x\pi(b), \\
& {}_A \la x, y \ra=xy^* \, , \quad \la x, y \ra_B=\pi^{-1}(x^* y) .
\end{align*}
By easy computations, we can see that $X_{\pi}$ is an $A-B$-equivalence bimodule.
We call $X_{\pi}$ an $A-B$-equivalence bimodule induced by $\pi$. Let $\alpha$ be an automorphism
of $A$. Then in the same way as above, we construct $X_{\alpha}$, as $A-B$-equivalence bimodule.
Let $u_{\alpha}$ be a unitary element in $M(A\rtimes_{\alpha}\BZ)$ implementing $\alpha$.
Hence $\alpha=\Ad(u_{\alpha})$. We regard $Au_{\alpha}$ as an $A-A$-equivalence bimodule as follows:
\begin{align*}
& a\cdot xu_{\alpha}=axu_{\alpha}\, , \quad xu_{\alpha}\cdot a =x\alpha(a), \\
& {}_A \la xu_{\alpha} \, , \, yu_{\alpha} \ra \, , \quad \la xu_{\alpha} \, , \, yu_{\alpha} \ra_A
=\alpha^{-1}(x^* y)
\end{align*}
for any $a, x, y\in A$.

\begin{lemma}\label{lem:pre1} With the above notation, $X_{\alpha}\cong Au_{\alpha}$ as
$A-A$-equivalence bimodules.
\end{lemma}
\begin{proof}
This is immediate by easy computations.
\end{proof}

Let $A$ be a $C^*$-algebra and $X$ an $A-A$-equivalence bimodule. Let $A\rtimes_X \BZ$ be the
crossed product of $A$ by $X$ defined in \cite {AEE:crossed}.
We regard the $C^*$-algebra $\BK$ as the trivial $\BK-\BK$-equivalence bimodule. Then we obtain
an $A\otimes\BK-A\otimes\BK$-equivalence bimodule $X\otimes\BK$ and we can also consider the
the crossed product
$$
(A\otimes\BK)\rtimes_{X\otimes\BK}\BZ
$$
of $A\otimes\BK$ by $X\otimes\BK$. Hence we have the following inclusions of $C^*$-algebras:
$$
A\subset A\rtimes_X \BZ \, , \quad A\otimes\BK\subset (A\otimes\BK)\rtimes_{X\otimes\BK}\BZ .
$$
Since there is an isomorphism $\pi$ of $(A\otimes\BK)\rtimes_{X\otimes\BK}\BZ$ onto
$(A\rtimes_X \BZ)\otimes\BK$ such that $\pi |_{A\rtimes\BK}=\id$ on $A\otimes\BK$, we identify
$A\otimes\BK \subset (A\otimes\BK)\rtimes_{X\otimes\BK}\BZ$ with
$A\otimes\BK\subset (A\rtimes_X \BZ)\otimes\BK$. Thus $A\subset A\rtimes_X \BZ$ and 
$A\otimes\BK\subset (A\otimes\BK)\rtimes_{X\otimes\BK}\BZ $ are strongly Morita equivalent.
Let $H_A$ be the $A\otimes\BK-A$-equivalence bimodule defined as follows:
Let $H_A =(A\otimes\BK)(1_{M(A)}\otimes e_{11})$ as $\BC$-vector spaces. For any $a\in A$, $k\in \BK$,
$x, y\in A\otimes\BK$,
\begin{align*}
& (a\otimes k)\cdot x(1\otimes e_{11})=(a\otimes k)x(1\otimes e_{11}) \, , \\
& x(1\otimes e_{11})\cdot a =x(a\otimes e_{11}) , \\
& {}_{A\otimes\BK} \la  x(1\otimes e_{11}) \, , \, y(1\otimes e_{11}) \ra =x(1\otimes e_{11})y^* \, , \\
& \la x(1\otimes e_{11}) \, , \, y(1\otimes e_{11}) \ra_A =(1\otimes e_{11})x^* y(1\otimes e_{11}) ,
\end{align*}
where we identify $A$ with $A\otimes e_{11}$. Let $B$ be a $C^*$-algebra. Let $H_B$ be as
above.

\begin{lemma}\label{lem:pre2} With the above notation, let $M$ be an $A-B$-
equivalence bimodule. Then
$$
(1\otimes e_{11})\cdot (M\otimes\BK)\cdot (1\otimes e_{11})\cong M
$$
as $A-B$-equivalence bimodules.
\end{lemma}
\begin{proof} Since the linear span of the set
$$
\{x\otimes e_{ij} \, | \, x\in M \, , \, i, j\in \BN \}
$$
is dense in $M\otimes\BK$, $(1\otimes e_{11})\cdot (M\otimes\BK)\cdot (1\otimes e_{11})
\cong M\otimes e_{11}$ as $A-B$-equivalence bimodules. Hence
$$
(1\otimes e_{11})\cdot (M\otimes\BK)\cdot (1\otimes e_{11})\cong M
$$
as $A-B$-equivalence bimodules.
\end{proof}

\begin{lemma}\label{lem:pre3} With the above notation, let $M$ be an $A-B$-equivalence
bimodule. Then
$$
\widetilde{H_A}\otimes_{A\otimes\BK}(M\otimes\BK)\otimes_{B\otimes\BK}H_B \cong M
$$
as $A-B$-equivalence bimodules.
\end{lemma}
\begin{proof}
Let $\pi$ be the map from
$\widetilde{H_A}\otimes_{A\otimes\BK}(M\otimes\BK)\otimes_{B\otimes\BK}H_B$ to
$(1\otimes e_{11})\cdot (M\otimes\BK)\cdot (1\otimes e_{11})$ defined by
$$
\pi([a(1\otimes e_{11})]^{\widetilde{}}\otimes x\otimes b(1\otimes e_{11}))
=(1\otimes e_{11})\cdot (a^* \cdot x\cdot b)\cdot (1\otimes e_{11})
$$
for any $a\in A\otimes \BK$, $b\in B\otimes\BK$, $x\in M\otimes\BK$. Then by
easy computations, $\pi$ is an $A-B$-equivalence bimodule isomorphism of
$\widetilde{H_A}\otimes_{A\otimes\BK}(M\otimes\BK)\otimes_{B\otimes \BK}H_B$
onto $(1\otimes e_{11})\cdot (M\otimes\BK)\cdot (1\otimes e_{11})$. Thus by
Lemma \ref{lem:pre2}
$$
\widetilde{H_A}\otimes_{A\otimes\BK}(M\otimes\BK)\otimes_{B\otimes\BK}H_B \cong M
$$
as $A-B$-equivalence bimodules. 
\end{proof}
We prepare the following lemma which is applied in the 
next section.

\begin{lemma}\label{lem:pre4} Let $A$ and $B$ be $C^*$algebras and $X$ and $Y$ an
$A-A$-equivalence bimodule and $B-B$-equivalence bimodule, respectively. Let
$A\subset A\rtimes_X \BZ$ and $B\subset B\rtimes_Y \BZ$ be inclusions of $C^*$algebras
induced by $X$ and $Y$, respectively. We suppose that there is an $A-B$-equivalence bimodule
$M$ such that $Y\cong \widetilde{M}\otimes_A X\otimes_A M$ or
$\widetilde{Y}\cong \widetilde{M}\otimes_A X\otimes_A M$ as $B-B$-equivalence bimodules.
Then there is an $A\rtimes_X \BZ -B\rtimes_Y \BZ$-equivalence bimodule $N$ satisfying the following:
\newline
$(1)$ $M$ is included in $N$ as a closed subspace,
\newline
$(2)$ $A\subset A\rtimes_X \BZ$ and $B\subset B\rtimes_Y \BZ$ are strongly Morita equivalent
with respect to $N$ and its closed subspace $M$. 
\end{lemma}
\begin{proof} Modifying the proof of \cite [Theorem 4.2]{AEE:crossed},
we prove this lemma. We suppose that $Y\cong \widetilde{M}\otimes_A X\otimes_A M$ as
$B-B$-equivalence bimodules. Let $L_M$ be the linking $C^*$-algebra for $M$ defined by
$$
L_M =\begin{bmatrix} A & M \\
\widetilde{M} & B \end{bmatrix} .
$$
Also, let $W$ be the $L_M -L_M$- equivalence bimodule defined in the proof of
\cite [Theorem 4.2]{AEE:crossed}, which is defined by
$$
W=\begin{bmatrix} X & X\otimes_A M \\
Y\otimes_B \widetilde{M} & Y \end{bmatrix} .
$$
Let $L_M \rtimes_W \BZ$ be the crossed product of $L_M$ by $W$ and let
$$
p=\begin{bmatrix} 1_{M(A)} & 0 \\
0 & 0 \end{bmatrix} \, , \quad
q=\begin{bmatrix} 0 & 0 \\
0 & 1_{M(B)} \end{bmatrix} .
$$
Furthermore, let $N=p(L_M \rtimes_W \BZ)q$. Then since $M=pL_M q$, $M$ is a closed subspace of $N$. 
Hence by the proof of \cite [Theorem 4.2]{AEE:crossed}, $A\subset A\rtimes_X \BZ$ and
$B\subset B\rtimes_Y \BZ$ are strongly Morita equivalent with respect $N$ and its closed subspace $M$.
\par
Next, we suppose that $\widetilde{Y}\cong \widetilde{M}\otimes_A X\otimes_A M$. Let
$$
W_0 =\begin{bmatrix} X & X\otimes_A M \\
\widetilde{Y}\otimes_B \widetilde{M} & \widetilde{Y} \end{bmatrix} .
$$
Then $W_0$ is an $L_M -L_M$-equivalence bimodule. Let $N_0 =p(L_M \rtimes_{W_0}\BZ)q$.
By the above discussions, $A\subset A\rtimes_X \BZ$ and $B\subset B\rtimes_{\widetilde{Y}}\BZ$
are strongly Morita equivalent with respect to $N_o$ and its closed subspace $M$. On the other hand,
there is an isomorphism $\pi$ of $B\rtimes_Y \BZ$ onto $B\rtimes_{\widetilde{Y}}\BZ$ such that
$\pi|_{B}=\id$ on $B$. Let $X_{\pi}$ be the $B\rtimes_{\widetilde{Y}}\BZ-B\rtimes_Y \BZ$
-equivalence bimodule induced by $\pi$.
Then $B$ is a closed subspace of $X_{\pi}$ and we regard $B$ as the trivial $B-B$-equivalence bimodule
since $\pi|_{B} =\id$ on $B$. Thus $A\subset A\rtimes_X \BZ$ and $B\subset B\rtimes_Y \BZ$ are
strongly Morita equivalent with respect to $N_0 \otimes_{B\rtimes_{\widetilde{Y}}\BZ}X_{\pi}$ and
its closed subspace $M\otimes_B B(\cong M)$. Therefore, we obtain the conclusion.
\end{proof}

\begin{lemma}\label{lem:pre5} With the above notation, we suppose that $A$ is a $\sigma$-unital
$C^*$-algebra. Then there is an automorphism $\alpha$ of $A\otimes\BK$ such that
$A\otimes\BK\subset (A\otimes\BK)\rtimes_{\alpha}\BZ$ is isomorphic to $A\otimes\BK\subset
(A\otimes\BK)\rtimes_{X\otimes\BK}\BZ$ as inclusions of $C^*$-algebras.
\end{lemma}
\begin{proof} Since $A$ is $\sigma$-unital, by Brown, Green and Rieffel \cite [Corollary 3.5]
{BGR:linking}, there is an automorphism $\alpha$ of $A\otimes\BK$ such that $X\otimes\BK\cong
X_{\alpha}$ as $A\otimes\BK -A\otimes\BK$-equivalence bimodules, where $X_{\alpha}$ is
the $A\otimes\BK -A\otimes\BK$-equivalence bimodule induced by $\alpha$.
Let $u_{\alpha}$ be a unitary element in
$M((A\otimes\BK)\rtimes_{\alpha}\BZ)$ implementing $\alpha$. We regard $(A\otimes\BK)u_{\alpha}$
as an $A\otimes\BK -A\otimes\BK$-equivalence bimodule as above. Then by Lemma \ref{lem:pre1},
$X_{\alpha}\cong (A\otimes\BK)u_{\alpha}$ as $A\otimes\BK -A\otimes\BK$-equivalence bimodules.
Let $(A\otimes\BK)\rtimes_{(A\otimes\BK)u_{\alpha}}\BZ$ be the crossed product of $A\otimes\BK$ by
$(A\otimes\BK)u_{\alpha}$. Then by the definition of the crossed product of a $C^*$-algebra by
an equivalence bimodule, we can see that 
$$
(A\otimes\BK)\rtimes_{\alpha}\BZ\cong (A\otimes\BK)\rtimes_{(A\otimes\BK)u_{\alpha}}\BZ
$$
as $C^*$-algebras. Since $X_{\alpha}\cong X\otimes\BK$ as $A\otimes\BK-A\otimes\BK$-
equivalence bimodules, we obtain that
$$
(A\otimes\BK)\rtimes_{(A\otimes\BK)u_{\alpha}}\BZ\cong (A\otimes\BK)\rtimes_{X_{\alpha}}\BZ
\cong (A\otimes\BK)\rtimes_{X\otimes\BK}\BZ
$$
as $C^*$-algebras. Since the above isomorphisms leave any element in $A\otimes\BK$ invariant,
we can see that $A\otimes\BK\subset (A\otimes\BK)\rtimes_{\alpha} \BZ$ is isomorphic to
$A\otimes\BK\subset(A\otimes\BK)\rtimes_{X\otimes\BK}\BZ$ as inclusions of $C^*$-algebras.
\end{proof}

\section{Strong Morita equivalence}\label{sec:SM} Let $A$ and $B$ be $\sigma$-unital $C^*$-algebras
and $X$ and $Y$ an $A-A$-equivalence bimodule and a $B-B$-equivalence bimodule, respectively.
Let $A\subset A\rtimes_X \BZ$ and $B\subset B\rtimes_Y \BZ$ be the inclusions of $C^*$-algebras
induced by $X$ and $Y$, respectively. We suppose that $A\subset A\rtimes_X \BZ$ and
$B\subset B\rtimes_Y \BZ$ are strongly Morita equivalent with respect to an
$A\rtimes_X \BZ -B\rtimes_Y \BZ$-equivalence bimodule $N$ and its closed subspace $M$.
We suppose that $A' \cap M(A\rtimes_X \BZ)=\BC 1$. Then since the inclusion
$A\otimes\BK\subset (A\otimes\BK)\rtimes_{X\otimes\BK}\BZ$ is isomorphic to the inclusion
$A\otimes\BK\subset (A\rtimes_X \BZ)\otimes\BK$ as inclusions of $C^*$-algebras,
by \cite [Lemma 3.1]{Kodaka:countable},
$$
(A\otimes\BK)' \cap((A\otimes\BK)\rtimes_{X\otimes\BK}\BZ)=\BC 1 .
$$
Also, by the above assumptions, the inclusion
$A\otimes\BK\subset (A\otimes\BK)\rtimes_{X\otimes\BK}\BZ$ is strongly Morita equivalent to
the inclusion $B\otimes\BK\subset (B\otimes\BK)\rtimes_{Y\otimes\BK}\BZ$ with respect to the
$(A\otimes\BK)\rtimes_{X\otimes\BK}\BZ-(B\otimes\BK)\rtimes_{Y\otimes\BK}\BZ$-equivalence
bimodule $N\otimes\BK$ and its closed subspace $M\otimes\BK$. By Lemma \ref {lem:pre5},
there are an automorphism $\alpha$ of $A\otimes\BK$ and an automorphism $\beta$ of $B\otimes\BK$
such that $A\otimes\BK\subset (A\otimes\BK)\rtimes_{X\otimes\BK}\BZ$ and
$B\otimes\BK\subset (B\otimes\BK)\rtimes_{Y\otimes\BK}\BZ$ are isomorphic to
$A\otimes\BK\subset (A\otimes\BK)\rtimes_{\alpha}\BZ$ and
$B\otimes\BK\subset (B\otimes\BK)\rtimes_{\beta}\BZ$ as inclusions of $C^*$-algebras, respectively.
Hence we can assume that $A\otimes\BK\subset(A\otimes\BK)\rtimes_{\alpha}\BZ$ and
$B\otimes\BK\subset (B\otimes\BK)\rtimes_{\beta}\BZ$ are strongly Morita equivalent with respect to
an $(A\otimes\BK)\rtimes_{\alpha}\BZ -(B\otimes\BK)\rtimes_{\beta}\BZ$-equivalence bimodule
$N\otimes\BK$ and its closed subspace $M\otimes\BK$. Since $A$ and $B$ are
$\sigma$-unital, in the same way as in the proof of \cite [Proposition 3.5]{Kodaka:Picard2} or
\cite [Proposition 3.1]{BGR:linking}, there is an isomorphism $\theta$ of $(B\otimes\BK)\rtimes_{\beta}\BZ$
onto $(A\otimes\BK)\rtimes_{\alpha}\BZ$ satisfying the following:
\newline
(1) $\theta|_{B\otimes\BK}$ is an isomorphism of $B\otimes\BK$ onto $A\otimes\BK$,
\newline
(2) There is an $(A\otimes\BK)\rtimes_{\alpha}\BZ-(B\otimes\BK)\rtimes_{\beta}\BZ$-
equivalence bimodule isomorphism $\Phi$ of $N\otimes\BK$ onto $Y_{\theta}$ such that
$\Phi|_{M\otimes\BK}$ is an $A\otimes\BK-B\otimes\BK$-equivalence bimodule isomorphism of
$M\otimes\BK$ onto $X_{\theta}$,
where $Y_{\theta}$ is the $(A\otimes\BK)\rtimes_{\alpha}\BZ-
(B\otimes\BK)\rtimes_{\beta}\BZ$-equivalence bimodule induced by $\theta$ and $X_{\theta}$ is the
$A\otimes\BK-B\otimes\BK$-equivalence bimodule induced by $\theta|_{B\otimes\BK}$.
\par
Let
$$
\gamma=\theta|_{B\otimes\BK}\circ\beta\circ\theta|_{B\otimes\BK}^{-1}
$$
and let $\lambda$ be the linear automorphism of $X_{\theta}$ defined by $\lambda(x)=\gamma(x)$
for any $x\in X_{\theta}(=A\otimes\BK)$.

\begin{lemma}\label{lem:SM1} With the above notation, $\gamma$ and $\beta$ are strongly Morita
equivalent with respect to $(X_{\theta}, \lambda)$.
\end{lemma}
\begin{proof} For any $x, y\in X_{\theta}$,
\begin{align*}
{}_{A\otimes\BK}\la \lambda(x) \, , \, \lambda(y) \ra 
& =\gamma(xy^* )=\gamma({}_{A\otimes\BK} \la x \, , \, y \ra) , \\
\la \lambda(x) \, ,\, \lambda(y) \ra _{B\otimes\BK} & =\theta|_{B\otimes\BK}^{-1}(\gamma(x^* y))
=\beta(\theta|_{B\otimes\BK}^{-1}(x^* y))=\beta(\la x \, \, y \ra_{B\otimes\BK}) .
\end{align*}
Hence $\gamma$ and $\beta$ are strongly Morita equivalent with respect to $(X_{\theta}, \lambda)$.
\end{proof}

By the proof of \cite [Theorem 5.5]{Kodaka:countable}, there is an automorphism $\phi$ of
$\BZ$ satisfying that $\gamma^{\phi}$ and $\alpha$ are exterior equivalent, that is, there is a unitary
element $z\in M(A\otimes\BK)$ such that
$$
\gamma^{\phi}=\Ad(z)\circ \alpha\, , \quad \underline{\alpha}(z)=z ,
$$
where $\gamma^{\phi}$ is the automorphism of $A\otimes\BK$ induced by $\gamma$ and $\phi$,
that is, $\gamma^{\phi}$ is defined by $\gamma^{\phi}=\gamma^{\phi(1)}$.
We note that $\gamma^{\phi}=\gamma$ or $\gamma^{\phi}=\gamma^{-1}$.
We regard $A\otimes\BK$ as the trivial $A\otimes\BK-A\otimes\BK$-equivalence bimdule.
Let $\mu$ be the linear automorphism of $A\otimes\BK$ defined by
$$
\mu(x)=\alpha(x)z^*
$$
for any $x\in A\otimes\BK$.

\begin{lemma}\label{lem:SM2} With the aboe notation, $\alpha$ and $\gamma^{\phi}$ are strongly Morita
equivalent with respect to $(A\otimes\BK,  \mu)$.
\end{lemma}
\begin{proof} For any $x, y\in A\otimes\BK$,
\begin{align*}
{}_{A\otimes\BK} \la \mu(x) \, , \, \mu(y) \ra & ={}_{A\otimes\BK} \la \alpha(x)z^* \, , \, \alpha(y)z^* \ra
=\alpha(xy^* )=\alpha({}_{A\otimes\BK} \la x \, , \, y \ra) , \\
\la \mu(x) \, , \, \mu(y) \ra_{A\otimes\BK} & =z\alpha(x^* y)z^* =\gamma^{\phi}(x^* y)
=\gamma^{\phi}(\la x \, ,\, y \ra_{A\otimes\BK}) .
\end{align*}
Therefore, we obtain the conclusion.
\end{proof}
Let $\nu$ be the linear automorphism of $X_{\theta}$ defined by
$$
\nu(x)=\gamma^{\phi}(z^* x)
$$
for any $x\in X_{\theta}(=A\otimes\BK)$.

\begin{lemma}\label{lem:SM3} With the above notation, $\alpha$ and $\beta^{\phi}$ are
strongly Morita equivalent with respect to $(X_{\theta}, \nu)$, where $\beta^{\phi}$ is the automorphism
of $B\otimes\BK$ induced by $\beta$ and $\phi$.
\end{lemma}
\begin{proof} For any $x, y\in X_{\theta}$,
\begin{align*}
{}_{A\otimes\BK} \la \nu(x) \, , \, \nu(y) \ra & =
{}_{A\otimes\BK} \la \gamma^{\phi}(z^* x) \, , \, \gamma^{\phi}(z^* y) \ra=\gamma^{\phi}(z^* xy^* z)
=z\alpha(z^* xy^* z)z^* \\
& =\alpha(xy^* )=\alpha({}_{A\otimes\BK} \la x \, , \, y \ra) , \\
\la \nu(x) \, , \, \nu(y) \ra_{B\otimes\BK} & =
 \la \gamma^{\phi}(z^* x) \, , \, \gamma^{\phi}(z^* y) \ra_{B\otimes\BK}
 =\theta|_{B\otimes\BK}^{-1}(\gamma^{\phi}(x^* y))=\beta^{\phi}(\theta|_{B\otimes\BK}^{-1}(x^* y)) \\
& =\beta^{\phi}(\la x \, , \, y \ra_{B\otimes\BK}) .
\end{align*}
Therefore, we obtain the conclusion.
\end{proof}
Since $\beta^{\phi}=\beta$ or $\beta^{\phi}=\beta^{-1}$, by Lemma \ref{lem:SM3}, $\alpha$ is
strongly Morita equivalent to $\beta$ or $\beta^{-1}$.
\par
(I) We suppose that $\alpha$ is strongly Morita equivalent to $\beta$. Then by Lemma \ref {lem:SM3},
there is the linear automorphism $\nu$ of $X_{\theta}$ satisfying the following:
\newline
(1) $\nu(a\cdot x)=\alpha(a)\cdot \nu(x)$,
\newline
(2) $\nu(x\cdot b)=\nu(x)\cdot \beta(b)$,
\newline
(3) ${}_{A\otimes\BK} \la \nu(x) \, , \, \nu(y)\ra=\alpha({}_{A\otimes\BK} \la x \, ,\, y \ra)$,
\newline
(4) $\la \nu(x) \, , \, \nu(y)\ra_{B\otimes\BK}=\beta(\la x \, ,\, y \ra_{B\otimes\BK})$,
\newline
for any $a\in A\otimes\BK$, $b\in B\otimes\BK$, $x, y\in X_{\theta}$.

\begin{lemma}\label{lem:SM4} With the above notation and assumptions, let $X_{\alpha}$ and $X_{\beta}$ be
the $A\otimes\BK-A\otimes\BK$-equivalence bimodule and the $B\otimes\BK-B\otimes\BK$-equivalence
bimodule induced by $\alpha$ and $\beta$, respectively. Then
$$
X_{\beta}\cong\widetilde{X_{\theta}}\otimes_{A\otimes\BK}X_{\alpha}\otimes_{A\otimes\BK}X_{\theta}
$$
as $B\otimes\BK-B\otimes\BK$-equivalence bimodules.
\end{lemma}
\begin{proof} Let $\Psi$ be the map from
$\widetilde{X_{\theta}}\otimes_{A\otimes\BK}X_{\alpha}\otimes_{A\otimes\BK}X_{\theta}$ to
$X_{\beta}$ defined by
$$
\Psi(\widetilde{x}\otimes a\otimes y)=\la x\, , \, a\cdot \nu(y) \ra_{B\otimes\BK}
$$
for any $x, y\in X_{\theta}$, $a\in X_{\alpha}$.
Then for any $x, x_1 , y, y_1 \in X_{\theta}$, $a, a_1 \in X_{\alpha}$,
\begin{align*}
{}_{B\otimes\BK} \la \widetilde{x}\otimes a\otimes y \, , \, \widetilde{x_1}\otimes a_1 \otimes y_1 \ra
& ={}_{B\otimes\BK} \la \widetilde{x}\cdot {}_{A\otimes\BK} \la a\otimes y \, , \, a_1 \otimes y_1 \ra \, , \,
\widetilde{x_1} \ra \\
& =\la \, {}_{A\otimes\BK} \la a_1 \otimes y_1 \, , \, a \otimes y \ra \cdot x \, , \, x_1 \ra_{B\otimes\BK} \\
& =\la \, {}_{A\otimes\BK} \la a_1 \cdot {}_{A\otimes\BK} \la y_ 1 \, , \, y \ra \, , \, a \ra \cdot x \, , \, x_1
\ra_{B\otimes\BK} \\
& =\la \, {}_{A\otimes\BK} \la a_1 \alpha({}_{A\otimes\BK} \la y_1 \, , \, y \ra) \, , \, a \ra
\cdot x \, , \, x_1 \ra_{B\otimes\BK} \\
& =\la a_1 \alpha({}_{A\otimes\BK} \la y_1 \, ,\, y \ra)a^* \cdot x \, , \, x_1 \ra_{B\otimes\BK} .
\end{align*}
On the other hand,
\begin{align*}
{}_{B\otimes\BK} \la \Psi(\widetilde{x}\otimes a\otimes y) \, , \,
\Psi(\widetilde{x_1}\otimes a_1 \otimes y_1 )\ra
& ={}_{B\otimes\BK} \la \la x\, , \, a\cdot \nu(y) \ra_{B\otimes\BK} \, , \,
\la x_1 \, , \, a_1 \cdot \nu(y_1 ) \ra_{B\otimes\BK} \ra \\
& =\la x\, , \, a\cdot\nu(y) \ra_{B\otimes\BK}\, \, \la a_1 \cdot \nu(y_1 ) \, , \, x_1 \ra_{B\otimes\BK} \\
& =\la x\, , \, a\cdot \nu(y)\cdot \la a_1 \cdot \nu(y_1 ) \, , \, x_1 \ra_{B\otimes\BK} \, \, \ra_{B\otimes\BK} \\
& =\la x \, , \, {}_{A\otimes\BK} \la a\cdot \nu(y) \, , \, a_1 \cdot \nu(y_1 ) \ra \cdot x_1 \ra_{B\otimes\BK} \\
& =\la x\, , \, a  \, {}_{A\otimes\BK} \la \nu(y) \, , \, \nu(y_1 ) \ra a_1^* \cdot x_1 \ra_{B\otimes\BK} \\
& =\la a_1 \, {}_{A\otimes\BK} \la \nu(y_1 ) \, , \, \nu(y) \ra a^* \cdot x\, , \, x_1 \ra_{B\otimes\BK} \\
& =\la a_1 \alpha({}_{A\otimes\BK} \la y_1 \, , \, y \ra)a^* \cdot x \, , \, x_1 \ra_{B\otimes\BK} .
\end{align*}
Hence the map $\Psi$ preserves the left $B\otimes\BK$-valued inner products. Also,
\begin{align*}
\la \widetilde{x}\otimes a\otimes y \, , \, \widetilde{x_1}\otimes a_1 \otimes y_1 \ra_{B\otimes\BK}
& =\la y \, , \, \la \widetilde{x}\otimes a \, , \, \widetilde{x_1 }\otimes a_1 \ra_{A\otimes\BK}\cdot y_1
\ra_{B\otimes\BK} \\
& =\la y \, \, \, \la a \, , \, {}_{A\otimes\BK} \la x\, , \, x_1 \ra \cdot a_1 \ra_{A\otimes\BK}\cdot y_1
\ra_{B\otimes\BK} \\
& =\la y \, \, \, \la a \, , \, {}_{A\otimes\BK} \la x\, , \, x_1 \ra a_1 \ra_{A\otimes\BK}\cdot y_1
\ra_{B\otimes\BK} \\
& =\la y\, , \, \alpha^{-1}(a^* \, {}_{A\otimes\BK} \la x\, , \, x_1 \ra a_1 )\cdot y_1 \ra_{B\otimes\BK} .
\end{align*}
On the other hand,
\begin{align*}
\la \Psi(\widetilde{x}\otimes a\otimes y)\, , \, \Psi(\widetilde{x_1}\otimes a_1 \otimes y_1 ) \ra_{B\otimes\BK}
& =\la \la x\, , \, a\cdot\nu(y) \ra_{B\otimes\BK}\, , \, \la x_1 \, , \, a_1 \cdot\nu(y_1 ) \ra_{B\otimes\BK}
\ra_{B\otimes\BK} \\
& =\beta^{-1}(\la a\cdot\nu(y) \, , \, x \ra_{B\otimes\BK}\,
\la x_1 \, , \, a_1 \cdot \nu(y_1 )\ra_{B\otimes\BK}) \\
& =\beta^{-1}(\la a\cdot\nu(y) \, , \, x\cdot\la x_1 \, , \, a_1 \cdot \nu(y_1 )
\ra_{B\otimes\BK} \, \ra_{B\otimes\BK}) \\
& =\beta^{-1}(\la a\cdot\nu(y) \, , \, {}_{A\otimes\BK} \la x \, , \, x_1 \ra a_1 \cdot\nu(y_1 )
\ra_{B\otimes\BK}) \\
& =\beta^{-1}(\la \nu(y) \, , \, a^* {}_{A\otimes\BK} \la x \, ,\, x_1 \ra a_1 \cdot \nu(y_1 ) \ra_{B\otimes\BK}) \\
& =\la y \, , \, \alpha^{-1}(a^* \, {}_{A\otimes\BK} \la x \, , , x_1 \ra a_1 )\cdot y_1 \ra_{B\otimes\BK} .
\end{align*}
Hence the map $\Psi$ preserves the right $B\otimes\BK$-valued inner products. Therefore, we obtain the
conclusion.
\end{proof}

(II) We suppose that $\alpha$ is strongly Morita equivalent to $\beta^{-1}$. Then by Lemma \ref{lem:SM4},
$$
X_{\beta^{-1}}\cong \widetilde{X_{\theta}}\otimes_{A\otimes\BK}X_{\alpha}\otimes_{A\otimes\BK}X_{\theta}
$$
as $B\otimes\BK-B\otimes\BK$-equivalence bimodules. Thus we obtain the following:

\begin{lemma}\label{le:SM5} With the above notation and assumptions,
$$
\widetilde{X_{\beta}}\cong X_{\beta^{-1}}
\cong\widetilde{X_{\theta}}\otimes_{A\otimes\BK} X_{\alpha}\otimes_{A\otimes\BK}X_{\theta}
$$
as $B\otimes\BK-B\otimes\BK$-equivalence bimodules. 
\end{lemma}

We recall that there is an $(A\otimes\BK)\rtimes_{\alpha}\BZ- (B\otimes\BK)\rtimes_{\beta}\BZ$-
equivalence bimodule isomorphism $\Phi$ of $N\otimes\BK$ onto $Y_{\theta}$ such that
$\Phi|_{M\otimes\BK}$ is an $A\otimes\BK-B\otimes\BK$-equivalence bimodule isomorphism of
$M\otimes\BK$ onto $X_{\theta}$. We identify $M\otimes\BK$ with $X_{\theta}$ by $\Phi|_{M\otimes\BK}$.
Then
$$
X_{\beta} \cong \widetilde{(M\otimes\BK)}\otimes_{A\otimes\BK}X_{\alpha}\otimes_{A\otimes\BK}
(M\otimes\BK)
$$
or
$$
\widetilde{X_{\beta}}\cong(\widetilde{M\otimes\BK)}\otimes_{A\otimes\BK}X_{\alpha}
\otimes_{A\otimes\BK}(M\otimes\BK)
$$
as $B\otimes\BK-B\otimes\BK$-equivalence bimodules. Also, we recall that $X_{\alpha}\cong X\otimes\BK$
as $A\otimes\BK-A\otimes\BK$-equivalence bimodules and that $X_{\beta}\cong Y\otimes\BK$ as
$B\otimes\BK-B\otimes\BK$-equivalence bimodules. Thus
$$
Y\otimes\BK\cong \widetilde{(M\otimes\BK)}\otimes_{A\otimes\BK}(X\otimes\BK)
\otimes_{A\otimes\BK}(M\otimes\BK)
$$
or
$$
\widetilde{Y\otimes\BK}\cong(\widetilde{M\otimes\BK)}\otimes_{A\otimes\BK}(X\otimes\BK)
\otimes_{A\otimes\BK}(M\otimes\BK)
$$
as $B\otimes\BK-B\otimes\BK$-equivalence bimodules. Furthermore, by Lemma \ref{lem:pre3}
$$
X\otimes\BK\cong H_A \otimes_A X \otimes_A \widetilde{H_A}
$$
as $A\otimes\BK-A\otimes\BK$-equivalence bimodues and
$$
Y\otimes\BK\cong H_B \otimes_B Y\otimes_B \widetilde{H_B}
$$
as $B\otimes\BK-B\otimes\BK$-equivalence bimodues. Thus
$$
Y\cong\widetilde{H_B}\otimes_{B\otimes\BK}(\widetilde{M\otimes\BK})\otimes_{A\otimes\BK} H_A \otimes_A
X\otimes_A \widetilde{H_A}\otimes_{A\otimes\BK}(M\otimes\BK)\otimes_{B\otimes\BK}H_B
$$
or
$$
\widetilde{Y}\cong\widetilde{H_B}\otimes_{B\otimes\BK}(\widetilde{M\otimes\BK})
\otimes_{A\otimes\BK} H_A \otimes_A
X\otimes_A \widetilde{H_A}\otimes_{A\otimes\BK}(M\otimes\BK)\otimes_{B\otimes\BK}H_B
$$
as $B\otimes\BK-B\otimes\BK$-equivalence bimodues. Also, by Lemma \ref{lem:pre3},
$$
\widetilde{H_A}\otimes_{A\otimes\BK}(M\otimes\BK)\otimes_{B\otimes\BK} H_B \cong M
$$
$A-B$-equivalence bimodules. Hence
$$
Y\cong\widetilde{M}\otimes_A X\otimes_A M \quad \text{or}\quad
\widetilde{Y}\cong\widetilde{M}\otimes_A X\otimes_A M
$$
as $B-B$-equivalence bimodules. Therefore, we obtain the following:

\begin{thm}\label{thm:SM6} Let $A$ and $B$ be $\sigma$-unital $C^*$-algebras and $X$ and $Y$
an $A-A$-equivalence bimodule and a $B-B$-equivalence bimodule, respectively. Let
$A\subset A\rtimes_X \BZ$ and $B\subset B\rtimes_Y \BZ$ be the inclusions of $C^*$-algebras
induced by $X$ and $Y$, respectively. We suppose that $A' \cap M(A\rtimes_X \BZ)\cong \BC1$.
Then the following conditions are equivalent:
\newline
$(1)$ $A\subset A\rtimes_X \BZ$ and $B\subset B\rtimes_Y \BZ$ are strongly Morita equivalent
with respect to an $A\rtimes_X \BZ -B\rtimes_Y \BZ$-equivalence bimodule $N$ and its closed
subspace $M$,
\newline
$(2)$ $Y\cong \widetilde{M}\otimes_A X \otimes_A M$ or 
$\widetilde{Y}\cong \widetilde{M}\otimes_A X \otimes_A M$
as $B-B$-equivalence bimodules.
\end{thm}
\begin{proof} (1)$\Rightarrow$(2): This is immediate by the above discussions.
(2)$\Rightarrow$(1): This is immediate by Lemma \ref{lem:pre4}.
\end{proof}

\section{The Picard groups}\label{sec:Picard} Let $A$ be a unital $C^*$-algebra and
$X$ an $A-A$-equivalence bimodule. Let $A\subset A\rtimes_X \BZ$ be the inclusion of
unital $C^*$-algebra induced by $X$. We suppose that $A' \cap(A\rtimes_X \BZ)=\BC 1$.
In this section, we shall compute $\Pic(A, A\rtimes_X \BZ)$, the Picard group of the inclusion
$A\subset A\rtimes_X \BZ$ (See \cite {Kodaka:Picard2}).
\par
Let $G$ be the subgroup of $\Pic(A)$ defined by
\begin{align*}
G=\{[M]\in\Pic(A) \, | \, X\cong\widetilde{M}\otimes_A X\otimes_A M \, \text{or} \,
\widetilde{X} & \cong\widetilde{M}\otimes_A X\otimes_A M \\
& \text{as $A-A$-equivalence bimodules}\}
\end{align*}
Let $f_A$ be the homomorphism of $\Pic(A, A\rtimes_X \BZ)$ to $\Pic(A)$
defined by
$$
f_A ([M, N])=[M]
$$
for any $[M, N]\in\Pic(A, A\rtimes_X \BZ)$. First, we show $\Ima f_A=G$, where $\Ima f_A$ is the image
of $f_A$.

\begin{lemma}\label{lem:P1} With the above notation, $\Ima f_A =G$.
\end{lemma}
\begin{proof} Let $[M, N]\in\Pic(A, A\rtimes_X \BZ)$. Then by the definition of $\Pic(A, A\rtimes_X \BZ)$,
the inclusion $A\subset A\rtimes_X \BZ$ is strongly Morita equivalent to itself with respect to an
$A\rtimes_X \BZ -A\rtimes_X \BZ$-equivalence bmodule $N$ and its closed subspace $M$.
Hence by Theorem \ref{thm:SM6}, $X\cong\widetilde{M}\otimes_A X\otimes_A M$ or
$\widetilde{X}\cong\widetilde{M}\otimes_A X\otimes_A M$ as $A-A$-equivalence bimodules. Thus
$\Ima f_A \subset G$. Next, let $[M]\in G$. Then by Lemma \ref{lem:pre4}, there is an $A\rtimes_X \BZ-
A\rtimes_X \BZ$-equivalence bimodule $N$ satisfying following:
\newline
(1) $M$ is included in $N$ as a closed subspace,
\newline
(2) $[M, N]\in\Pic(A, A\rtimes_X \BZ)$.
\newline
Hence $G\subset \Ima f_A$. Therefore, we obtain the conclusion.
\end{proof}

Next, we compute $\Ker f_A$, the kernel of $f_A$. Let $\Aut (A, A\rtimes_X \BZ)$ be the group of
all automorphisms $\alpha$ of $A\rtimes_X \BZ$ such that $\alpha|_A$ is an automorphism of $A$.
Let $\Aut_0 (A, A\rtimes_X \BZ)$ be the group of all automorphisms $\alpha$ of $A\rtimes_X \BZ$ such
that $\alpha|_A =\id$ on $A$. It is clear that $\Aut_0 (A, A\rtimes_X \BZ)$ is a normal subgroup of
$\Aut (A, A\rtimes_X \BZ)$. Let $\pi$ be the homomorphism of $\Aut(A, A\rtimes_X \BZ)$ to
$\Pic(A, A\rtimes_X \BZ)$ defined by
$$
\pi(\alpha)=[M_{\alpha}, N_{\alpha}]
$$
for any $\alpha\in \Aut(A, A\rtimes_X \BZ)$, where $[M_{\alpha}, N_{\alpha}]$ is an element in
$\Pic(A, A\rtimes_X \BZ)$ induced by $\alpha$ (See \cite [Section 3]{Kodaka:Picard2}).

\begin{lemma}\label{lem:P2} With the above notation,
$$
\Ker f_A =\{[A, N_{\beta}]\in\Pic(A, A\rtimes_X \BZ)  \, | \, \beta\in\Aut_0 (A, A\rtimes_X \BZ) \} .
$$
\end{lemma}
\begin{proof} Let $[M, N]\in\Ker f_A$. Then $[M]=[A]$ in $\Pic(A)$ and by \cite [Lemma 7.5]{Kodaka:Picard2},
there is a $\beta\in\Aut_0 (A, A\rtimes_X \BZ)$ such that
$$
[M, N]=[A, N_{\beta}]
$$
in $\Pic(A, A\rtimes_X \BZ)$, where $N_{\beta}$ is the $A\rtimes_X \BZ -A\rtimes_X \BZ$-equivalence
bimodule induced by $\beta$. Therefore, we obtain the conclusion.
\end{proof}

Let $\Int(A, A\rtimes_X \BZ)$ be the group of all $\Ad(u)$ such that $u$ is a unitary element in $A$.
By \cite [Lemma 3.4]{Kodaka:Picard2},
$$
\Ker \, \pi \cap\Aut_0 (A, A\rtimes_X\BZ)=\Int(A, A\rtimes_X \BZ)\cap \Aut_0 (A, A\rtimes_X \BZ) .
$$
Hence
\begin{align*}
& \Ker \, \pi\cap\Aut_0 (A, A\rtimes_X \BZ) \\
& =\{\Ad(u)\in\Aut_0 (A, A\rtimes_X \BZ) \, | \, \text{$u$ is a unitary
element in $A$}\} \\
& =\{\Ad(u)\in\Aut_0 (A, A\rtimes_X \BZ) \, | \, \text{$u$ is a unitary
element in $A' \cap A$}\} .
\end{align*}
Since $A' \cap (A\rtimes_X \BZ)=\BC1$, $A' \cap A=\BC1$. Thus
$$
\Ker \, \pi \cap\Aut_0 (A, A\rtimes_X \BZ)=\{1\}.
$$
It follows that we can obtain the following lemma:

\begin{lemma}\label{lem:P3} With the above notation, $\Ker f_A \cong \Aut_0 (A, A\rtimes_X \BZ)$.
\end{lemma}
\begin{proof} Since $\Ker \, \pi\cap\Aut_0 (A, A\rtimes_X \BZ)=\{1\}$, by Lemma \ref{lem:P2}
\begin{align*}
\Ker f_A & =\pi(\Aut_0 (A, A\rtimes_X \BZ))\cong\Aut_0 (A, A\rtimes_X \BZ)/
(\Ker \, \pi\cap\Aut_0 (A, A\rtimes_X \BZ) )\\
& =\Aut_0 (A, A\rtimes_X \BZ) .
\end{align*}
Therefore, we obtain the conclusion.
\end{proof}

We recall that the inclusions of $C^*$-algebras $A\otimes\BK\subset (A\rtimes_X \BZ)\otimes\BK$ and
$A\otimes\BK\subset (A\otimes\BK)\rtimes_{X\otimes\BK}\BZ$ are isomorphic as inclusions of $C^*$-algebras.
Also, there is an automorphism $\alpha$ of $A\otimes\BK$ such that
$A\otimes\BK\subset (A\otimes\BK)\rtimes_{X\otimes\BK}\BZ$ and
$A\otimes\BK\subset (A\otimes\BK)\rtimes_{\alpha}\BZ$ are isomorphic as inclusions of $C^*$-algebras.

\begin{lemma}\label{lem:P4} With the above notation, $\alpha$ is free, that is, for any $n\in\BZ\setminus\{0\}$,
$\alpha^n$ satisfies the following: If $x\in M(A\otimes\BK)$ satisfies that $xa=\alpha^n (a)x$ for any
$a\in A\otimes\BK$, then $x=0$.
\end{lemma}
\begin{proof} Since $A' \cap (A\rtimes_X \BZ)=\BC1$, by \cite [Lemma 3.1]{Kodaka:countable}
$(A\otimes\BK)' \cap M((A\rtimes_X \BZ)\otimes\BK)=\BC1$. Hence since $A\otimes\BK\subset
(A\rtimes_X \BZ)\otimes\BK$ is isomorphic to $A\otimes\BK\subset(A\otimes\BK)\rtimes_{\alpha}\BZ$
as inclusions
of $C^*$-algebras,
$$
(A\otimes\BK)' \cap M((A\otimes\BK)\rtimes_{\alpha}\BZ)=\BC1 .
$$Thus by \cite [Corollary 4.2]{Kodaka:countable}, $\alpha$ is free.
\end{proof}

For any $n\in\BZ$, let $\delta_n$ be the function on $\BZ$ defined by
$$
\delta_n (m) =\begin{cases} 1 & m=n \\
0 & m\ne n \end{cases} .
$$
We regard $\delta_n$ as an element in $M((A\otimes\BK)\rtimes_{\alpha}\BZ)$. Let $E^{M(A\otimes\BK)}$
be the faithful conditional expectation from $M(A\otimes\BK)\rtimes_{\underline{\alpha}}\BZ$ onto
$M(A\otimes\BK)$ defined in B\'edos and Conti \cite [Section 3]{BC:discrete}. Let
$E^{A\otimes\BK}=E^{M(A\otimes\BK)}|_{(A\otimes\BK)\rtimes_{\alpha}\BZ}$. Let $\{u_i \}_{i\in I}$ be
an approximate unit of $A\otimes\BK$. We fix the approximate unit $\{u_i \}_{i\in I}$ of $A\otimes\BK$.
For any $x\in M((A\otimes\BK)\rtimes_{\alpha}\BZ)$, we define the Fourier coefficient of $x$ at $n\in\BZ$
and the Fourier series of $x$ as in \cite [Section 2]{Kodaka:countable}. We show that
$\Aut_0 (A\otimes\BK, (A\otimes\BK)\rtimes_{\alpha}\BZ)\cong\BT$.
\par
Let $\beta\in \Aut_0 (A\otimes\BK, (A\otimes\BK)\rtimes_{\alpha}\BZ)$. For any $a\in A\otimes\BK$,
$$
\underline{\beta}(\delta_1 )a\underline{\beta}(\delta_1^* )=\beta(\delta_1 a \delta_1^* )=\beta(\alpha(a))
=\alpha(a) .
$$
Hence $\underline{\beta}(\delta_1 )a=\alpha(a)\underline{\beta}(\delta_1 )$ for any $a\in A\otimes\BK$.

\begin{lemma}\label{lem:P5} Let $\sum_{n\in\BZ}a_n \delta_1^n =\sum_{n\in\BZ}a_n \delta_n$ be the
Fourier series of $\underline{\beta}(\delta_1 )$, where $a_n \in M(A\otimes\BK)$ for any $n\in \BZ$.
Then for any $n\in \BZ$, $a\in A\otimes\BK$,
$$
a_n \alpha^{n-1}(a)=aa_n .
$$
\end{lemma}
\begin{proof} For any $a\in A\otimes\BK$, $\underline{\beta}(\delta_1 )a=\alpha(a)\underline{\beta}(\delta_1 )$.
Hence the Fourier series of $\underline{\beta}(\delta_1 )a$ is:
$$
\sum_{n\in\BZ}a_n \alpha^n (a)\delta_1^n .
$$
Also, the Fourier series of $\alpha(a)\underline{\beta}(\delta_1 )$ is:
$$
\sum_{n\in\BZ}\alpha(a)a_n \delta_1^n .
$$
Thus $a_n \alpha^n (a)=\alpha(a)a_n $ for any $a\in A\otimes\BK$, $n\in \BZ$.
Therefore, we obtain the conclusion.
\end{proof}

\begin{lemma}\label{lem:P6} With the above notation,
$$
\Aut_0 (A\otimes\BK, (A\otimes\BK)\rtimes_{\alpha}\BZ)\cong\BT .
$$
\end{lemma}
\begin{proof} Let $\beta\in \Aut_0 (A\otimes\BK, (A\otimes\BK)\rtimes_{\alpha}\BZ)$ and
let $\sum_{n\in\BZ}a_n \delta_1^n$ be the Fourier series of $\underline{\beta}(\delta_1 )$.
Then by Lemma \ref{lem:P5}, $a_n \alpha^{n-1}(a)=aa_n$ for any $a\in A\otimes\BK$, $n\in \BZ$. Since
$\alpha$ is free by Lemma \ref {lem:P4}, $a_n =0$ for any $n\in\BZ\setminus\{1\}$. Thus
$\underline{\beta}(\delta_1 )=a_1 \delta_1$.
Since $\underline{\beta}(\delta_1 )a\underline{\beta}(\delta_1^* )=\alpha(a)$ for any $a\in A\otimes\BK$,
$$
a_1 \delta_1 a \delta_1^* a_1^* =\alpha(a) .
$$
Since $\delta_1 a\delta_1^* =\alpha(a)$,
$$
a_1 \alpha(a)a_1^* =\alpha(a)
$$
for any $a\in A\otimes\BK$. Since $\delta_1$ and $\underline{\beta}(\delta_1 )$ are unitary elements in
$M((A\otimes\BK)\rtimes_{\alpha}\BZ)$, $a_1$ is  a unitary element in $M(A\otimes\BK)$. Thus
$$
a_1 \alpha(a)=\alpha(a)a_1
$$
for any $a\in A\otimes\BK$. Since $(A\otimes\BK)' \cap M(A\otimes\BK)=\BC1$,
$a_1 \in \BC1$. Since $a_1$ is a unitary element in $M(A\otimes\BK)$, there is the unique element
$c_{\beta}\in\BT$, such that $a_1 =c_{\beta}1$. Let $\epsilon$ be the map from
$\Aut_0 (A\otimes\BK, (A\otimes\BK)\rtimes_{\alpha}\BZ)$ onto $\BT$ defined by $\epsilon(\beta)=c_{\beta}$.
By routine computations, we can see that $\epsilon$ is an isomorphism of
$\Aut_0 (A\otimes\BK, (A\otimes\BK)\rtimes_{\alpha}\BZ)$ onto $\BT$.
\end{proof}

\begin{lemma}\label{lem:P7} With the above notation,
$$
\Aut_0 (A, A\rtimes_X \BZ)\cong\Aut_0 (A\otimes\BK, (A\otimes\BK)\rtimes_{\alpha}\BZ) .
$$
\end{lemma}
\begin{proof} Since $A\otimes\BK\subset(A\otimes\BK)\rtimes_{\alpha}\BZ$ and
$A\otimes\BK\subset(A\rtimes_X\BZ)\otimes\BK$ are isomorphic as inclusions of $C^*$-algebras,
it suffices to show that
$$
\Aut_0 (A, A\rtimes_X \BZ)\cong \Aut_0 (A\otimes\BK, (A\rtimes_X \BZ)\otimes\BK) .
$$
Let $\kappa$ be the homomorphism of $\Aut_0 (A, A\rtimes_X \BZ)$ to
$\Aut_0 (A\otimes\BK, (A\rtimes_X \BZ)\otimes\BK)$ defined by
$$
\kappa(\beta)=\beta\otimes\id_{\BK}
$$
for any $\beta\in\Aut_0 (A, A\rtimes_X \BZ)$. Then it is clear that $\kappa$ is a monomorphism of
$\Aut_0 (A, A\rtimes_X \BZ)$ to $\Aut_0 (A\otimes\BK, (A\rtimes_X \BZ)\otimes\BK)$.
We show that $\kappa$ is surjective. Let $\gamma\in \Aut_0 (A\otimes\BK, (A\rtimes_X \BZ)\otimes\BK)$.
Then
$$
\gamma(a\otimes e_{ij})=a\otimes e_{ij}
$$
for any $a\in A$, $i, j \in \BN$. Thus
$$
\gamma(x\otimes e_{11})=(1\otimes e_{11})\gamma(x\otimes e_{11})(1\otimes e_{11})
$$
for any $x\in A\rtimes_X \BZ$. Hence there is an automorphism $\alpha$ of $A\rtimes_X \BZ$ such that
$$
\gamma(x\otimes e_{11})=\alpha(x)\otimes e_{11}
$$
for any $x\in A\rtimes_X \BZ$. For any $i, j\in \BN$, $x\in A\rtimes_X \BZ$,
\begin{align*}
\gamma(x\otimes e_{ij}) & =\gamma((1\otimes e_{11})(x\otimes e_{11})(1\otimes e_{1j}))
=(1\otimes e_{i1})(\alpha(x)\otimes e_{11})(1\otimes e_{1j}) \\
& =\alpha(x)\otimes e_{ij} .
\end{align*}
Especially , if $a\in A$, $\alpha(a)\otimes e_{ij}=\gamma(a\otimes e_{ij})=a\otimes e_{ij}$ for any $i, j\in \BN$.
Thus $\alpha(a)=a$ for any $a\in A$. Therefore, $\gamma=\alpha\otimes\id_{\BK}$ and
$\alpha\in \Aut_0 (A, A\rtimes_X \BK)$. Hence we can see that $\kappa$ is surjective.
\end{proof}

\begin{lemma}\label{lem:P8} With the above notation, $\Ker f_A \cong \BT$.
\end{lemma}
\begin{proof} This is immediate by Lemmas \ref {lem:P3}, \ref {lem:P6} and \ref {lem:P7}.
\end{proof}

By Lemmas \ref {lem:P1} and \ref {lem:P8}, we have the following exact sequence:
$$
1\longrightarrow\BT\longrightarrow\Pic(A, A\rtimes_X \BZ)\longrightarrow G\longrightarrow 1,
$$
where
\begin{align*}
G=\{[M]\in\Pic(A) \, | \, X\cong\widetilde{M}\otimes_A X\otimes_A M \, \text{or} \,
\widetilde{X} & \cong\widetilde{M}\otimes_A X\otimes_A M \\
& \text{as $A-A$-equivalence bimodules}\}
\end{align*}
Let $g$ be the map from $G$ to $\Pic(A, A\rtimes_X \BZ)$ defined by
$$
g([M])=[M, N] ,
$$
where $N$ is the $A\rtimes_X \BZ -A\rtimes_X \BZ$-equivalence bimodule defined in the proof of
Lemma \ref {lem:pre4}. Then $g$ is a homomorphism of $G$ to $\Pic(A, A\rtimes_X \BZ)$ such that
$$
f_A \circ g=\id
$$
on $G$. Thus we obtain the following:

\begin{thm}\label{thm:P9} Let $A$ be a unital $C^*$-algebra and $X$ an $A-A$-equivalence bimodule.
Let $A\subset A\rtimes_X \BZ$ be the unital inclusion of unital $C^*$-algebras induced by $X$.
We suppose that $A' \cap (A\rtimes_X \BZ)=\BC1$. Then $\Pic(A, A\rtimes_X \BZ)$ is a semi-direct
product of $G$ by $\BT$.
\end{thm}

\end{document}